\documentclass[11pt]{article}  %%
\usepackage{amssymb, amscd, amsmath, amsthm}

\newcommand{\beq}{\begin{equation}}
\newcommand{\eeq}{\end{equation}}
\newtheorem{lemma}{Lemma}
\newtheorem{corollary}{Corollary}

\allowdisplaybreaks

\title{On the Density Hypothesis for the Selberg Class}

\author{by\\
J\'anos Pintz\thanks{Supported by the National Research Development and Innovation Office of Hungary, NKFIH, K133819 and K147153.}}

\date{}

\begin{document}
\maketitle

\setcounter{equation}{0}
\numberwithin{equation}{section}

\setcounter{equation}{0}
\numberwithin{equation}{section}

\footnotetext{Keywords and phrases: Selberg class, density hypothesis.}

\footnotetext{2020 Mathematics Subject Classification: Primary 11M41.}

\section{Introduction}
\label{sec:1}

The Selberg class $S$ of $\mathcal L$-functions was introduced by Atle Selberg in 1992 \cite{Sel1992}.
The elements of the class $S$ are Dirichlet series $F(s)$ satisfying the following axioms:

\begin{itemize}
\item[(i)] $F(s)$ is absolutely convergent for $\sigma > 1$;

\item[(ii)] $(s - 1)^m F(s)$ is an entire function of finite order with an integer $m \geq 0$;

\item[(iii)] $F(s)$ satisfies a functional equation of the form
$$
\Phi(s) = \omega \overline\Phi(1 - s),
$$
where $|\omega| = 1$, $\overline f(s) = \overline{f(s)}$ and
$$
\Phi(s) = Q^s \prod_{j = 1}^r \Gamma(\lambda_j s + \mu_j)F(s), \ \ Q > 0, \ \lambda_j > 0, \ \text{\rm Re }\mu_j \geq 0;
$$
\item[(iv)] the Dirichlet coefficients $a(n)$ of $F(s)$ satisfy the Ramanujan condition $a(n) \ll n^\varepsilon$ for every $\varepsilon > 0$;
\item[(v)] $\log F(s)$ is a Dirichlet series, $\log F(s) = \sum\limits_p \log F_p(s)$,
$$
\log F_p(s) = \sum_{m = 1}^\infty \frac{b(p^m)}{p^{ms}}, \ \ \ b(p^m) \ll (p^m)^\vartheta \ \text{ for some }\ \vartheta < 1/2.
$$
\end{itemize}

One of the most important notions is the degree $d_F$ of a function $F \in S$ defined by
\beq
\label{eq:1.1}
d_F = 2 \sum_{j = 1}^r \lambda_j,
\eeq
which is, in fact, an invariant of $F$.

One of the main goals of the theory would be to characterize elements of the Selberg class if their degree $d_F$ is given.
Selberg conjectured that the degree is always a non-negative integer.
One of the other main goals (clearly hopeless at present) would be to show the Riemann Hypothesis for all $F \in S$, i.e., that all non-trivial zeros of $F$ lie on the line $\text{\rm Re }s = 1/2$.

Concerning the characterization problem Conrey and Ghosh \cite{CG1993} showed that there is no element $F$ with a degree $d_F \in (0, 1)$ while the only function with $d_F = 0$ (i.e. without $\Gamma$-factors) is $F = 1$.

Very deep results were reached by Kaczorowski and Perelli (\cite{KP1999} and \cite{KP2011}).
In \cite{KP1999} they showed that the only functions with $d_F = 1$ are the Riemann zeta and ordinary Dirichlet $\mathcal L$-functions.
In \cite{KP2011} they showed that there are no elements $F\in S$ with $1 < d_F < 2$.

For the number $N_F(T)$ of non-trivial zeros of $F$ with $0 \leq I ms \leq T$ Selberg showed that
\beq
\label{eq:1.2}
N_F(T) = d_F \frac{T(\log T + C)}{2\pi} + O(\log T),
\eeq
similarly to the case of Riemann's $\zeta(s)$.

As an approximation to the Riemann Hypothesis Carlson proved more than hundred years ago \cite{Car1920} that
\beq
\label{eq:1.3}
N(\sigma, T) \ll T^{4 \sigma(1 - \sigma)} \log^c T.
\eeq

The best possible (eventually uniform) conjecture of type
\beq
\label{eq:1.4}
N(1 - \eta, T) \ll T^{A(\eta)\eta} \log^c T, \ \ \ (\eta \leq 1/2)
\eeq
or
\beq
\label{eq:1.5}
N(1 - \eta, T) \ll_\varepsilon T^{A(\eta)\eta + \varepsilon} \ \text{ for any } \varepsilon > 0
\eeq
is by \eqref{eq:1.2} with $A(\eta) \leq A = 2$.
It is called the Density Hypothesis.

Although there were many improvements of Carlson's result in the past hundred years, the Density Hypothesis is still open.
A breakthrough result was shown half a century later by Hal\'asz and Tur\'an \cite{HT1969} who could show its validity in a fixed strip $\eta < c_1$ of the critical strip.
Several mathematicians showed Carlson type density theorems for elements of the Selberg class, including Kaczorowski and Perelli \cite{KP2003} who showed for any $F \in S$
\beq
\label{eq:1.6}
N_F(1 - \eta, T) \ll_\varepsilon T^{A_F(\eta)\eta + \varepsilon}
\eeq
with
\beq
\label{eq:1.7}
A_F(\eta) = 4(d_F + 3) \ \ \text{ for }\ \eta < 1/4.
\eeq

Other works proved even the corresponding Density Hypothesis for several elements $F \in S$.

Our present goal is to show a density theorem, namely the following

\smallskip
\noindent
{\bf Theorem.} {\it Under the above notations \eqref{eq:1.6} holds with
\beq
\label{eq:1.8}
A_F(\eta) = \max (d_F, 2) \ \ \text{ \it for }\ \eta \leq 1/10.
\eeq}

\smallskip
In particular, we obtain the Density Hypothesis if $d_F \leq 2$.
We do not use the deep fact
(resp. conjecture) that there are no elements of the Selberg class, if $d_F$ is not a non-negative integer.

\section{Notation. Proof of Theorem. Preparation}
\label{sec:2}

We begin with some notation and a definition.

Let us assume that we have zeros $\varrho_j = \beta_j + i \gamma_j$ of $F(S)$ with
$\gamma_j \in [T/2, T]$, $T$ large, $\beta_j = 1 - \eta_j$, $n_j \leq \eta \leq 1/4$,
$\sigma = 1 - \eta$.
Suppose further that for $j \neq \nu$
\beq
\label{eq:2.1}
|\gamma_j - \gamma_\nu| \geq 3 \mathcal L^3 \ \ \text{ with } \ \ \mathcal L = \log T.
\eeq

Let $\varepsilon$ be a generic arbitrary small positive constant which might be different at different occurrences.

Let us choose the parameters $X$ and $Y$ as
\beq
\label{eq:2.2}
X = T^\varepsilon,
\eeq
\beq
\label{eq:2.3}
Y = T^{d_F/2 + \varepsilon}.
\eeq

\noindent
{\bf Definition.} The implicit constants in the $\vartheta$ and $\ll$ symbols might depend on $\varepsilon$ and $F(s)$.
A non-trivial zero $\varrho = \beta_j + i \gamma_j = 1 - \eta_j + i\gamma_j$ of $F(S)$ will be called an extreme right hand (eRH) zero if the rectangle (for $H = \mathcal L$)
\beq
\label{eq:2.4}
R_H(\varrho) := \left\{ \sigma + it; \, \sigma \geq \beta_j + \frac1{\mathcal L}, \, |t - \gamma_j| \leq H^2 \right\}
\eeq
is free of zeros of $F(s)$.

Starting from any zero $\varrho_0$ with $\beta_0 \geq 1/2$ we have two possibilities:

(i) \ $\varrho_0$ is an eRH zero;

(ii) we can find another zero $\varrho_1 = \beta_1 + i\gamma_1$ with $\beta_1 \geq \beta_0 + 1/\mathcal L$,
$|\gamma_1 - \gamma_0| \leq \mathcal L^2$.

In case (i) we are ready, in case (ii) we continue the same procedure with $\varrho_1$ in place of $\varrho_0$.
In such a way we arrive after at most $\lceil \mathcal L/2\rceil$ steps at an eRH zero $\varrho'$ with
\beq
\label{eq:2.5}
\beta' \geq \beta_0, \ \ |\gamma' - \gamma_0| \leq \mathcal L^3.
\eeq

The advantage of using an eRH zero $\varrho'$ instead of an arbitrary $\varrho$ in our counting procedure will be clear from the following

\begin{lemma}
\label{lem:1}
If for a point $s_0 = \sigma_0 + it_0 = 1 - \eta_0 + it_0$ with $\sigma_0 \geq 1/2$, $|t_0| \leq 2T$ the rectangle $R_H (s_0)$ defined by $H = \mathcal L$ and \eqref{eq:2.4} with $\varrho$ replaced by $s_0$ is zero-free, then
\beq
\label{eq:2.6}
F\left(\frac12 + it\right) \ll T^{\sigma_0 - 1/2 + \varepsilon} \ \text{ for } \ |t - t_0| \leq \mathcal L^2/2.
\eeq
\end{lemma}

\begin{proof}
Let $\frac3{\mathcal L} < \delta$ be a sufficiently small parameter to be determined later.
Let us use the Borel--Carath\'eodory theorem for $\log F(z)$ with the circles of radius $r = 2 - \sigma_0 - \delta/2$ and $2 - \sigma_0 - \delta$ and centre $2 + it_0$.
Then we have on the larger circle
\beq
\label{eq:2.7}
\text{\rm Re }\log F(z) = \log |F(z)| \ll \mathcal L.
\eeq

Hence, on the smaller circle by $F(s) = t^{0(1)}$ and axiom (v)
\beq
\label{eq:2.8}
|\log F(z)| \ll \frac{4 - 2 \sigma_0 - \delta}{\delta/2} \mathcal L + \frac{4 - 2\sigma_0 - 3\delta/2}{\delta/2} \bigl|\log F (2 + it_0)\bigr| \ll \frac{\mathcal L}{\delta}.
\eeq

Afterwards apply Hadamard's three circle theorem with circles $C_1, C_2, C_3$, centered at $1/\delta + it_0$ passing through the points $2 + it_0$, $\sigma_0 + 10\delta + it_0$, $\sigma_0 + \delta + it_0$,
i.e. with radii $r_1 = 1/\delta - 2$, $r_2 = 1/\delta - \sigma_0 - 10 \delta$, $r_3 = 1/\delta - \sigma_0 - \delta$.
Let us denote the maximum of $F(z)$ on these circles by $M_1$, $M_2$, $M_3$.
We have then
\beq
\label{eq:2.9}
M_2 \leq M_1^{1 - a} \cdot M_3^a
\eeq
where by $\log (1 + x) = x - \frac{x^2}{2}(1 + o(1))$ for $x \to 0$
\begin{align}
\label{eq:2.10}
a &= \log \frac{r_2}{r_1}\!\! \biggm/\!\! \log \frac{r_3}{r_1} = \log\! \left(\!1 \! + \frac{2\! -\! \sigma_0 - 10 \delta}{1/\delta - 2}\right) \!\!\biggm/\!\! \log\! \left(\!1\! + \frac{2 - \sigma_0 - \delta}{1/\delta - 2}\right)\\
&\leq \left(\frac{2 - \sigma_0 - 10\delta}{1/\delta - 2} \left(1 - \frac{2\delta}{5}\right)\!\right) \!\! \biggm/\!\! \left(\frac{2 - \sigma_0 - \delta}{1/\delta - 2} \left(1 - \frac{2\delta}{3}\right)\right) \nonumber\\
&\leq  \frac{2 - \sigma_0 - 10\delta}{2 - \sigma_0 - \delta} \left(1 + \frac{\delta}{3}\right). \nonumber
\end{align}

For the elements $F$ of the Selberg class we have by axiom (v)
\beq
\label{eq:2.11}
M_1 = \max_{z \in C_1} |\log F(z)| \ll 1,
\eeq
while the argument of the Borel--Carath\'eodory theorem, yields \eqref{eq:2.8} for every point of the circle $C_3$ (not only for $\sigma_0 + \delta + it_0$)
\beq
\label{eq:2.12}
M_3 = \max_{z \in C_3} |\log F(z)| \ll \frac{\mathcal L}{\delta}.
\eeq

Taking into account \eqref{eq:2.9}--\eqref{eq:2.10} we obtain
\begin{align}
\label{eq:2.13}
\left|\log F(\sigma_0 + 10\delta + it_0)\right| &\leq M_2 \ll \frac1{\delta} \mathcal L^{a} \\
&\ll \frac1{\delta} \mathcal L^{(1 - \frac{9\delta}{2 - \sigma_0}) (1 + \frac{\delta}{3})} \nonumber\\
&\ll \frac1{\delta} \mathcal L^{1 - 6\delta}.  \nonumber
\end{align}
Choosing $\delta = 1/\sqrt{\log\mathcal L}$
\beq
\label{eq:2.14}
\left|\log F(\sigma_0 + 10\delta + it_0)\right| \ll \sqrt{\log \mathcal L} \cdot \mathcal L e^{-6\sqrt{\log \mathcal L}} = o(\mathcal L).
\eeq

Hence, from the functional equation we obtain
\beq
\label{eq:2.15}
\left| \log F (1 - \sigma_0 - 10\delta + it_0)\right| \leq \frac{d_F}{2} (2\sigma_0 - 1 + o(1))\log T.
\eeq

If we replace in the definition $R_{\mathcal L}(s)$ the parameter $\mathcal L$ by $\mathcal L/2$ then the whole argument yielding \eqref{eq:2.14}--\eqref{eq:2.15} remains valid.
Therefore we have \eqref{eq:2.14}--\eqref{eq:2.15} if $t_0$ is replaced by an arbitrary $t^*$ with
\beq
\label{eq:2.16}
|t^* - t_0| \leq \mathcal L^2/4.
\eeq

So, we can now use Hadamard's three lines theorem for the function
\beq
\label{eq:2.17}
f(z) = F(z) e^{(z - it^*)^2}
\eeq
on the lines $\sigma_1 = \sigma_0 + 10\delta$, $\sigma_2 = 1/2$, $\sigma_3 = 1 - \sigma_0 - 10\delta$.

Let us denote the corresponding maximums by $M_1$, $M_2$, $M_3$.
First we note that as \eqref{eq:2.14}--\eqref{eq:2.15} are valid for $t_0$ replaced by $t^*$ we have
\beq
\label{eq:2.17m}
M_1 = \sup_t |f(\sigma_1 + it)| \ll |F(\sigma_1 + it^*)| \ll T^\varepsilon
\eeq
and
\beq
\label{eq:2.18}
M_3 = \sup_t |f(\sigma_3 + it)| \ll |F(\sigma_3 + it^*)| \ll T^{d_F(2\sigma_0 - 1)/2 +\varepsilon}.
\eeq

Consequently, by Hadamard's three lines theorem we have
\beq
\label{eq:2.19}
M_2 = \sup_t |f(1/2 + it)| \ll (M_1 M_3)^{1/2} \ll T^{d_F(\sigma_0 - 1/2)/2 + \varepsilon}.
\eeq

In particular we have for $t^* \in [t_0 - \mathcal L^2/2, t_0 + \mathcal L^2/2]$
\beq
\label{eq:2.20}
\left| F\left(\frac12 + it^*\right)\right| = \left|f\left(\frac12 + it^*\right)\right| e^{-1/4} \ll T^{d_F(\sigma_0 - 1/2)/2 + \varepsilon}.
\eeq
\end{proof}

\section{The zero detection method}
\label{sec:3}

We will use the now standard method of Montgomery to detect the zeros of $F(s)$ with $\beta \geq 3/4$ with slight modifications applied by Kaczorowski and Perelli \cite{KP2003} to prove \eqref{eq:1.6}--\eqref{eq:1.7}.
We will closely follow \cite{KP2003}, so we will be brief.
Until the end of \eqref{eq:3.5} these zeros can be arbitrary with \eqref{eq:2.1}, later on we suppose that they are eRH zeros (see Section \ref{sec:2}).

Denoting the $p$-th Euler factor of $F(s) = \sum\limits_{n = 1}^\infty a(n) n^{-s}$ by $F_p(s)$ with a $z = z(\varepsilon)$ to be chosen later we write
\beq
\label{eq:3.1}
F(s, z) := \prod_{p > z} F_p(s), \ \ M_X(s, z) := \sum_{\substack{n \leq X\\ (n, P(z)) = 1}} a^{-1}(n) n^{-s}.
\eeq
Since $F_p(s) \neq 0$ for $\sigma \geq 1/2$ (see Section 2 of \cite{KP1999}), the zeros of $F(s)$ and $F(s,z)$ coincide in the halfplane $\sigma \geq 1/2$.
For $\sigma > 1$ we have
\beq
\label{eq:3.2}
F(s, z) M_X(s, z) = 1 + \sum_{n > X} c(n, z, X)n^{-s},
\eeq
where by Lemma 1 of \cite{KP2003}
\beq
\label{eq:3.3}
c(n) = c(n, z, X) \ll n^\varepsilon.
\eeq

By the well-known Mellin transform we have for a zero $\varrho$ of $F(s)$
\begin{align}
\label{eq:3.4}
I(\varrho) :&= e^{-1/Y} + \sum_{n > X} c(n)n^{-\varrho} e^{-n/Y} = \frac1{2\pi i} \int\limits_{(2)}
F(\varrho + s, z) M_X(\varrho + s, z)Y^s \Gamma(s)ds\\
&= r(X, Y, \varrho) + \frac1{2\pi i} \int\limits_{(1/2 - \beta)} F(\varrho + s, z) M_X(\varrho + s, z)Y^s \Gamma(s),\nonumber
\end{align}
where $r(X, Y, \varrho)$ denotes the residue of the integrand at $s = 1 - \varrho$, since the integrand is regular at $s = 0$.
We have
\beq
\label{eq:3.5}
r(X, Y, \varrho) \ll \left(M_X(\varrho + s, z) Y^s \Gamma(s)\right)^{(m - 1)}_{s = 1 - \varrho} \ll T^\varepsilon X Y^{1 - \beta} e^{-T} = o(1).
\eeq

Further, we have by $b(n) \ll n^\vartheta$, $\vartheta < 1/2$, in case of an eRH zero $\varrho$ for $|u|\leq \mathcal L^2/2$ by \eqref{eq:2.20}
\begin{align}
\label{eq:3.6}
\left|F\left(\frac12 + i(u + \gamma) z\right) \right| &\leq \prod_{p \leq z} F_p^{-1} \left(\frac12 + i(u + \gamma)\right)F \left(\frac12 + i(u + \gamma)\right)\\
&\ll (|u + \gamma| + 2)^{(d_F/2)(\beta - 1/2) + \varepsilon}.
\nonumber
\end{align}

Due to the exponential decay of the $\Gamma$-function we can restrict the integral on the RHS of \eqref{eq:3.4} to the interval $\bigl[1/2 - \beta - i \mathcal L^2 /2$, $1/2 - \beta + i \mathcal L^2/2\bigr]$ and so we obtain from \eqref{eq:3.4}--\eqref{eq:3.6} and \eqref{eq:2.1}
\begin{align}
\label{eq:3.7}
1 + O\left(\frac1{Y}\right) + \sum_{x < n < Y\mathcal L^2} c(n)n^{-\varrho} e^{-n/Y}
&\ll T^{(d_F/2)(\beta - 1/2) + \varepsilon} Y^{1/2 - \beta} + o(1)\\
&= o(1).
\nonumber
\end{align}

This implies by partial summation by $\beta \geq 1 - \eta$
\beq
\label{eq:3.8}
\sum_{X < n < Y\mathcal L^2} c(n) n^{-(1 - \eta) - i\gamma} e^{-n/Y} \gg 1.
\eeq
Hence, by a dyadic subdivision of the interval $[X, Y\mathcal L^2]$ we obtain an interval
\beq
\label{eq:3.9}
[M, M'] \subseteq [M, 2M] \subset [X, Y\mathcal L^2]
\eeq
with
\beq
\label{eq:3.10}
\Bigl|\sum_{M < n \leq M'} c(n)n^{-(1 - \eta) - i\gamma_j} e^{-n/Y}\Bigr| \gg 1/\mathcal L
\eeq
if $\varrho = 1 - \eta_j - i\gamma_j$ was an eRH zero of $F(s)$ with $\eta_j \leq \eta$.

We will now use Hal\'asz's idea for a suitable $k$th power of the LHS of \eqref{eq:3.9} in a version of Heath-Brown \cite{Hea1979} which incorporates the twelfth power moment estimate of the Riemann zeta function due to Heath-Brown \cite{Hea1978}.

\begin{lemma}
\label{lem:2}
Suppose $t_1, \ldots, t_R$ are real numbers, $a_n$ arbitrary complex numbers, $\varepsilon > 0$
\beq
\label{eq:3.11}
|t_i - t_j| \geq 1 \ \text{ for } \ i, j \in [1, R], \  i \neq j, \ |t_i| \leq T.
\eeq
Then
\beq
\label{eq:3.12}
\sum_{r \leq R} \Bigl|\sum_N^{2N} a_n n^{-it_r}\Bigr|^2 \ll T^\varepsilon \bigl(N + R^{11/12} T^{1/6} N^{1/2}\bigr) \sum_N^{2N} |a_n|^2.
\eeq
\end{lemma}

\begin{corollary}
\label{cor:1}
Suppose that $s_r = 1 - \eta + it_r$ $(1 \leq r \leq R)$ with $t_r$ satisfying \eqref{eq:3.11}, $b_n$ arbitrary complex with $b_n \ll n^\varepsilon$.
Then
\beq
\label{eq:3.13}
\sum_{r \leq R} \Bigl|\sum_N^{2N} b_n n^{-s_r}\Bigr|^2 \ll T^\varepsilon N^{2\varepsilon} \bigl(N^{2\eta} + R^{11/12} T^{1/6} N^{-(1/2 - 2\eta)}\bigr).
\eeq
\end{corollary}

Now, if
\beq
\label{eq:3.14}
\Bigl|\sum_N^{2N} b_n n^{-s_r}\Bigr|^2 \gg T^{-\varepsilon},
\eeq
then we obtain from \eqref{eq:3.13}
\beq
\label{eq:3.15}
R \ll T^{3\varepsilon} N^{2\eta + 2\varepsilon}
\eeq
or
\beq
\label{eq:3.16}
R \ll R^{11/12} T^{1/6 + 3\varepsilon} N^{-(1/2 - 2\eta) + 2\varepsilon};
\eeq
consequently
\beq
\label{eq:3.17}
R \ll T^{2 + 36\varepsilon} N^{24\eta - 6 + 24\varepsilon}.
\eeq

Now, \eqref{eq:3.15} and \eqref{eq:3.17} mean that in case of
\beq
\label{eq:3.18}
N_0 := T^{\frac{2 - 2\eta}{6 - 24\eta} + C\varepsilon} \leq N \leq T^C
\eeq
we have (we remind that $\varepsilon$ is a generic constant)
\beq
\label{eq:3.19}
R \ll T^\varepsilon(N^{2\eta} + T^{2\eta}) \ \ \text{ for any }\ \varepsilon > 0.
\eeq

We consider now two cases according to the size of $N = M^k$:

(i) If $N_0 \leq N \ll Y^{1 + \varepsilon}$ then by \eqref{eq:3.19}
\beq
\label{eq:3.20}
R \ll T^\varepsilon (Y^{2\eta} + T^{2\eta}) = T^{d_F \eta + \varepsilon} + T^{2\eta + \varepsilon};
\eeq

(ii) if $N_0 \ll N \ll N_0^{3/2}$ then by \eqref{eq:3.19}
\beq
\label{eq:3.21}
R \ll T^\varepsilon (N_0^{3\eta} + T^{2\eta})
= T^{\frac{(1 - \eta)\eta}{(1 - 4\eta)} + \varepsilon} + T^{2\eta + \varepsilon} \ll T^{2\eta + \varepsilon}
\eeq
since $\eta \leq 1/7$.

If the original value of $M$ was less than $N_0^{1/2 + \varepsilon}$ we can find $k$ with
$M^k \in \bigl[N_0^{1 + \varepsilon}, N_0^{3/2 + \varepsilon}\bigr]$.
If $M \in \bigl[N_0^{1/2 + \varepsilon}, N_0^{1 + \varepsilon}\bigr]$, we choose $k = 2$ and obtain
\beq
\label{eq:3.22}
R \ll T^\varepsilon \bigl(N_0^{4\eta} + T^{2\eta}\bigr) \ll T^{\frac{4(1 - \eta)\eta}{3(1 - 4\eta)} + \varepsilon} + T^{2\eta + \varepsilon} \ll T^{2\eta + \varepsilon}
\eeq
since $\eta \leq 1/10$.

Finally, for $M \geq N_0^{1 + \varepsilon}$ we have \eqref{eq:3.20} by $M \leq Y$ with the choice $k = 1$.

For the sake of completeness we have to note that we used the fact that if $\sum\limits_M^{2M} a_m m^{-s}$ is a Dirichlet polynomial with $|a_m| \leq C (\delta) m^\delta$ then its $k$th power, i.e.,
\beq
\label{eq:3.23}
\sum_N^{2^kN} b_n n^{-s}
\eeq
satisfies
\begin{align}
\label{eq:3.24}
b_n = \Bigl|\sum_{\substack{n = n_1 n_2 \ldots n_k\\
n_i \in (M, 2M]}}
a_{n_1} a_{n_2} \ldots a_{n_k}\Bigr| &\leq \tau_k(n)C(\delta)^k(n_1\ldots n_R)^\delta\\
&\leq C(\delta)^k n^\delta \tau^k(n)\nonumber\\
&\ll_{\delta, k} n^{2\delta}
\nonumber\end{align}
where $\tau_k(n)$ is the generalized divisor function.
In our case we have $\tau(n) \ll n^{c/\log\log n}$, $\delta = \varepsilon$ and $k \leq \log_2(1/\varepsilon) + 1$
($\log_2 m$ denotes the logarithm of base $2$).

We note that in the course of proof we used that the number of eRH zeros satisfying \eqref{eq:2.1} is at most a factor $C\mathcal L^5$ times higher than the total number of all zeros with $|\gamma|\leq T$, $\beta \geq \sigma = 1 - \eta$.

{\small
\noindent
J\'anos Pintz\\
HUN-REN Alfr\'ed R\'enyi Institute of Mathematics \\
Budapest, Re\'altanoda u. 13--15\\
H-1053 Hungary\\
e-mail: pintz@renyi.hu}

\end{document}